\documentclass[11pt]{amsart}
\usepackage{amsmath,amsfonts,amssymb,amsthm}
\usepackage[alphabetic]{amsrefs}
\usepackage{hyperref}
\usepackage{graphicx,color}
\usepackage{pictexwd,dcpic,epsf}

\newtheorem{theorem}{Theorem}[section]
\newtheorem{lemma}[theorem]{Lemma}

\newtheorem*{mtheorem}{Main Theorem}

\theoremstyle{definition}

\newtheorem*{questions}{Questions}

\theoremstyle{remark} 
\newtheorem*{remark}{Remark}

\newcommand{\ov}{\overline}

\newcommand {\gi}{\mr{girth} \,}

\newcommand {\mr}{\mathrm}

\newcommand{\hthz}{\widehat{\Theta}}
\newcommand{\hth}[1]{\widehat{\Theta}_{#1}}
\newcommand{\hG}{\widehat{G}}
\newcommand{\hE}{\widehat{e}}
\newcommand{\htl}[1]{\widehat{l}_{#1}}
\newcommand{\phij}[2]{\varphi_{#1}^{#2}}
\newcommand{\qij}[2]{q_{#1}^{#2}}

\newcommand{\cW}[1]{\mathcal W^{#1}}

\begin{document}

\title{Residually finite non-exact groups}

\author{Damian Osajda}
\address{Instytut Matematyczny,
	Uniwersytet Wroc\l awski\\
	pl.\ Grun\-wal\-dzki 2/4,
	50--384 Wroc{\l}aw, Poland}
\address{Institute of Mathematics, Polish Academy of Sciences\\
	\'Sniadeckich 8, 00-656 War\-sza\-wa, Poland}
\address{Dept.\ of Math.\ \& Stats., McGill University\\ Montreal, Quebec, Canada H3A 0B9}
\email{dosaj@math.uni.wroc.pl}
\date{\today}

\begin{abstract}
We construct the first examples of residually finite non-exact groups. 
\end{abstract}

\maketitle

\section{Introduction}
\label{s:intro}

A finitely generated group is \emph{non-exact} if its reduced $C^{\ast}$--algebra is non-exact. 
Equivalently, it has no Guoliang Yu's property A (see e.g.\ \cite[Chapter 11.5]{Roe-book}).
Most classical groups are \emph{exact}, that is, are not non-exact. The first examples of non-exact
groups were the so-called \emph{Gromov monsters} \cite{Gromov2003}. In this paper we rely on author's construction of groups containing isometrically expanders \cite{Osajda-Monster}. The isometric embedding of an expanding family of graphs performed in the latter
construction is possible thanks to using a graphical small cancellation. That particular construction
is crucial for results in the current paper.
\begin{mtheorem}
	There exist finitely generated residually finite non-exact groups defined by infinite graphical small cancellation presentations.
\end{mtheorem}
This answers one of few questions from the Open Problems chapter of the Brown-Ozawa book \cite[Problem 10.4.6]{BrownOzawa-book}. Some motivations for the question can be found there.
Our interest in the problem is twofold: First, we plan to use residually finite non-exact groups 
constructed here for producing other, essentially new examples of non-exact groups; Second,
we believe that our examples might be useful for constructing and studying metric spaces with
interesting new coarse geometric features. More precisely, let $G$ be a finitely generated infinite residually finite group, and let $( N_i )_{i=1}^{\infty}$ be a sequence of its finite index normal subgroups with
$\bigcap_{i=1}^{\infty}N_i = \{ 1\}$. The \emph{box space} of $G$ corresponding to $(N_i)$ is  the coarse disjoint union $\bigsqcup_{i=1}^{\infty} G/N_i$, with each $G/N_i$ endowed with the word metric coming from a given finite generating set for $G$. Properties of the group $G$ are often related to coarse geometric properties of its box space. For example, a group is amenable
iff its box space has property A \cite[Proposition 11.39]{Roe-book}. Box spaces provide a powerful method for producing
metric spaces with interesting coarse geometric features (see e.g.\ \cite[Chapter 11.3]{Roe-book}).
The groups constructed in the current article open a way to studying  box spaces of non-exact groups and
make the following questions meaningful.
\begin{questions}
	What are coarse geometric properties of box spaces of non-exact groups? Can non-exactness of a group be characterized by coarse geometric properties of its box space?\footnote{After circulating the first version of the article I was informed that Thibault Pillon introduced a notion of ``fibred property A",
		and proved that a finitely generated residually finite group is exact iff its box space has this property (unpublished).}
\end{questions}

The idea of the construction of groups as in the Main Theorem is as follows. The group is defined by an infinite
graphical small cancellation presentation. It is a limit of a direct sequence of groups $G_i$ with surjective bonding maps -- each $G_i$ has a graphical small cancellation presentation being
a finite chunk of the infinite presentation. Such finite chunks are constructed inductively, using results of
\cite{Osajda-Monster}, so that they
satisfy the following conditions.
Each group $G_i$ is hyperbolic and acts geometrically on a CAT(0) cubical complex, hence it is residually finite.\footnote{Note that Pride~\cite{Pride1989} constructed infinitely presented classical small cancellation groups
	that are not residually finite. They are limits of hyperbolic CAT(0) cubical (hence residually finite) groups.} 
For every $i$, there exists a map $\phij{i}{}\colon G_i\to F_i$ to a finite group such that no nontrivial element of the $i$--ball around identity is mapped to $1$. Every $\phij{i}{}$ factors through the quotient maps $G_i\twoheadrightarrow G_j$ so that it induces a map of the limit group $G$ to a finite group injective on a large ball.
The residual finiteness of $G$ follows. Finally, $G$ is non-exact since its Cayley 
graph contains a sequence of graphs (relators) without property A. 

In Section~\ref{s:prel} we present preliminaries on graphical small cancellation presentations
and we recall some results from \cite{Osajda-Monster}. In Section~\ref{s:constr} we present the inductive construction of the infinite graphical small cancellation presentation proving the Main Theorem.
\medskip

\noindent
{\bf Acknowledgments.} I thank Ana Khukhro and Kang Li for suggesting the question.
I am grateful to them, and to Thibault Pillon for useful remarks. 
I thank Damian Sawicki for pointing out mistakes in the earlier version of the manuscript.
The author was partially supported by (Polish) Narodowe Centrum Nauki, grant no.\ UMO-2015/\-18/\-M/\-ST1/\-00050. The paper was written while visiting McGill University.
The author would like to thank the Department of Mathematics and Statistics of McGill University
for its hospitality during that stay.

\section{Preliminaries}
\label{s:prel}
We follow closely (up to the notation) \cite{Osajda-Monster}.

\subsection{Graphs}
\label{s:prelg}
All graphs considered in this paper are \emph{simplicial}, that is, they are undirected, have no
loops nor multiple edges. In particular, we will consider Cayley graphs of groups, denoted $\mr{Cay}(G,S)$ -- the Cayley graph of $G$ with respect to the generating set $S$.
For a set $S$, an $S$--\emph{labelling} of a graph $\Theta$ is the assignment
of elements of $S\cup S^{-1}$ ($S^{-1}$ being the set of formal inverses of elements of $S$) to directed edges (pairs of vertices) of $\Theta$, satisfying the following
condition: If $s$ is assigned to $(v,w)$ then $s^{-1}$ is assigned to $(w,v)$. All labellings considered in this paper are \emph{reduced}: If $s$ is assigned to $(v,w)$, and $s'$ is assigned to $(v,w')$ then
$s=s'$ iff $w=w'$. For a covering of graphs $p\colon \hthz \to \Theta$, having a labelling $(\Theta,l)$
we will always consider the \emph{induced} labelling $(\hthz,\widehat l)$: the label of an edge $e$
in $\hthz$ is the same as the label of $p(e)$.
Speaking about the metric on a connected graph $\Theta$ we mean the metric space $(\Theta^{(0)},d)$, where $\Theta^{(0)}$ is the set of vertices of $\Theta$ and $d$ is the path metric.
The \emph{ball} of radius $i$ around $v$ in $\Theta$ is $B_i(v,\Theta):= \{w\in \Theta^{(0)} \, | \, d(w,v) \leqslant i \}$. In particular, the metric on $\mr{Cay}(G,S)$ coincides with the word metric on $G$ given
by $S$.

\subsection{Graphical small cancellation}
\label{s:prelsc}
A \emph{graphical presentation} is the following data: $\mathcal{P}=\langle S \; | \; (\Theta,l)\rangle$,
where $S$ is a finite set -- the \emph{generating set}, and $(\Theta,l)$ is a graph $\Theta$ with
an $S$--labelling $l$. We assume that $\Theta$ is a disjoint (possibly infinite) union of finite connected
graphs without separating edges $(\Theta_{i})_{i\in I}$, and the labelling $l$ restricted to $\Theta_{i}$ is denoted by $l_i$.
We write $(\Theta,l)=(\Theta_i, l_i)_{i\in I}$.  
A graphical presentation $\mathcal{P}$ defines a group $G:=F(S)/R$, where $R$ is the normal closure 
in $F(S)$ of the subgroup generated by words in $S\cup S^{-1}$ read along (directed) loops in $\Theta$.

A \emph{piece} is a labelled path occurring in two distinct connected components $\Theta_i$ and $\Theta_j$,
or occurring in a single $\Theta_i$ in two places not differing by an automorphism of $(\Theta_i,l_i)$.
Observe that if $(\hth{i},\widehat l_{i})\to (\Theta_i,l_i)$ is a normal covering then two lifts of a path 
in $\Theta_i$ differ by a covering automorphism. In particular, if the covering corresponds to a characteristic subgroup of $\pi_1(\Theta_i)$ then a
lift of a non-piece is a non-piece.

For $\lambda \in (0,1/6]$, the labelling $(\Theta,l)$ or the presentation $\mathcal P$ are called
\emph{$C'(\lambda)$--small cancellation} if length of every piece appearing in $\Theta_i$ is strictly less
than $\lambda \mr{girth}(\Theta_i)$, where $\mr{girth}(\Theta_i)$ is the length of a shortest
simple cycle in $\Theta_i$. Such presentations define infinite groups. The introduction of graphical small cancellation is attributed to Gromov \cite{Gromov2003}. For more details see e.g.\
\cite{Wise-quasiconvex,Osajda-Monster}. We will use mostly results proven already
in \cite{Wise-quasiconvex,Osajda-Monster}, so we list only the most important features of groups defined
by graphical small cancellation presentations. 

First, observe that if $(\Theta,l)$ is a $C'(\lambda)$--small cancellation labelling, and $\hth{i} \to \Theta_i$ is a covering corresponding to a characteristic subgroup of $\pi_1(\Theta_i)$, for each $i$, then the induced
labelling $(\hthz,\widehat l)$ is also $C'(\lambda)$--small cancellation. The following result was first stated by Gromov.
\begin{lemma}[\cite{Gromov2003}]
	\label{l:Gromov}
	Let $G$ be the group defined by a graphical $C'(\lambda)$--small cancellation presentation $\mathcal P$, for $\lambda \in (0,1/6]$.
	Then, for every $i$, there is an isometric embedding $\Theta_i \to \mr{Cay}(G,S)$. 
\end{lemma}
The isometric embedding above is just an embedding of $S$--labelled graphs.

\subsection{Walls in graphs}
\label{s:walls}

A \emph{wall} in a connected graph is a collection  of edges such that removing all interiors of these
edges  decomposes the graph in exactly two connected components.
There are many ways for defining walls in finite graphs $\Theta_i$. We would like however that such walls
``extend" to walls in $\mr{Cay}(G,S)$. 
We use a particular construction of walls in finite graphs. For such graph $\Theta_i$, let $\hth{i}$ denote 
its $\mathbb{Z}_2$--homology cover, that is, a normal cover corresponding to the kernel of the obvious map
$\pi_1(\Theta_i,v_0) \to H_1(\Theta_i;\mathbb{Z}_2)$. Wise \cite{Wise-quasiconvex} observed that there is
a natural structure of walls on $\hth{i}$: for every edge of $\Theta_i$ its preimage is a wall in $\hth{i}$.
We call these walls \emph{$\mathbb{Z}_2$--homology cover walls}.
We will use results of Wise \cite{Wise-quasiconvex} to show that such walls, defined for $\hth{1},\ldots,\hth{i}$ give rise to walls in $\mr{Cay}(\hG,S)$, where 
$\hG$ is the group with the graphical presentation $\langle S \; | \; (\hth{1},\htl{1}),\ldots,(\hth{i},\htl{i}) \rangle$, for $\htl{j}$ being the labeling induced
from $l_j$. Furthermore, we show that $\hG$ acts geometrically on the associated CAT(0) cube complex.
We begin with a technical lemma. All the walls here are $\mathbb{Z}_2$--homology cover walls.

\begin{lemma}
	\label{l:l5.40}
	Let $\widehat{\gamma}$ be a geodesic in $\hth{i}$ whose first and last edges, $\hE_1$ and $\hE_2$, respectively, belong
	to walls $w_1$ and $w_2$. Suppose there is a wall $w$ such that one of its edges $\hE$ belongs to
	$\widehat{\gamma}$, and another edge $\hE'$ is $(\gi(\Theta_{i})/24)$--close to $\hE_2$. Then there is
	a wall $w'$ 	separating $\hE_1$ and $\hE_2$ but such that no edge of $w'$ is
	$(\gi(\Theta_{i})/24)$--close to an edge of $w_1$ or $w_2$.
\end{lemma}
\begin{proof}
	By definition $w_1$, $w_2$, and $w$ consist of preimages of edges $e_1, e_2$, and $e$ in $\Theta_i$, respectively.
	Since $\hE$ and $\hE'$ are in the same wall, but distinct, it follows that their distance is  
	at least $\gi(\Theta_{i})-1$. Hence the length of $\widehat{\gamma}$ is at least $(\gi(\Theta_{i})-1-\gi(\Theta_{i})/24)>\gi(\Theta_i)/2$. Consider a projection $\gamma$ of $\widehat{\gamma}$.
	It is a path (sequence of edges) of length (number of edges) at least 
	$\gi(\Theta_i)/2$, without back-tracks. Hence, there exists an edge
	$f$ in $\gamma$ that is not $(\gi(\Theta_{i})/24)$--close to $e_1$ or $e_2$ and that is 
	passed by $\gamma$ an odd number of times. The wall $w'$ defined by $f$, that is, the wall consisting of preimages
	of $f$ is the desired wall separating $\hE_1$ and $\hE_2$.  
\end{proof}

\begin{remark}
	The above lemma is needed to use \cite[Theorem 5.40 and Remark 5.41]{Wise-quasiconvex}\footnote{Here and everywhere we refer to the preprint version of \cite{Wise-quasiconvex} dated June 19, 2017.} in the sequel. Wise's work concerns the so-called \emph{cubical small cancellation}. It is a far-going generalization of the graphical small cancellation used in the current article (see e.g.\ \cite[3.s. Examples on p.\ 72 or Section 5.k.\ p.\ 124]{Wise-quasiconvex}). For example, in our case Wise's \emph{hyperplanes}
	reduce just to edges, and there exist only \emph{cone pieces}. Consequently, many assumptions appearing in 
	formulations of results in \cite{Wise-quasiconvex} are easily satisfied in the graphical small cancellation case.
\end{remark}

Let $\lambda \in (0,1/24]$, and let $(\Theta_1,l_1),\ldots,(\Theta_i,l_i)$ be a $C'(\lambda)$--small cancellation
labelling. Then we call the system of $\mathbb Z_2$--homology walls on $(\hth{1},\htl{1}),\ldots,$
$(\hth{i},\htl{i})$ a \emph{proper $\mathbb Z_2$--walling}.
The following lemma is our main tool for proving residual finiteness of intermediate steps in our construction.
\begin{lemma}
	\label{l:c0cc}
	Let $(\hth{1},\htl{1}),\ldots,
	(\hth{i},\htl{i})$ be equipped with a {proper $\mathbb Z_2$--walling}.
	Then the group $\hG=\langle S \; | \; (\hth{1},\htl{1}),\ldots,
	(\hth{i},\htl{i}) \rangle$ acts geometrically on a CAT(0) cubical complex.
\end{lemma}
\begin{proof}
	We use \cite[Theorem 5.40 and Remark 5.41]{Wise-quasiconvex} for proving that $\hG$ acts properly on a CAT(0) cubical complex. We verify that the graphical presentation $\mathcal P:=\langle S \; | \; (\hth{1},\htl{1}),\ldots,
	(\hth{i},\htl{i}) \rangle$ fulfills the conditions (1), (2), and (3) from Theorem 5.40 there.
	\medskip
	
	\noindent
	{\bf Condition (1) of \cite[Theorem 5.40]{Wise-quasiconvex}.} We have to show that $\mathcal P$ satisfies the \emph{generalized B(6) condition} of \cite[Definition 5.1]{Wise-quasiconvex} and has \emph{short innerpaths}. 
	
	\cite[Lemma 3.67]{Wise-quasiconvex} shows that $\mathcal P$ has short innerpaths.
	
	We now turn to \cite[Definition 5.1]{Wise-quasiconvex}. Points (1), (2), and (5) in this definition are obvious. For (3) and (4) observe that pieces in $\hth{j}$ have length at most $(\gi(\Theta_j)/24)$
	and distinct edges in the same wall in $\hth{j}$ are at distance at least $(\gi(\Theta_j)-1)$ (compare the proof of Lemma~\ref{l:l5.40} above).
	Hence, (3) and (4) follow as in \cite[Remark 5.2]{Wise-quasiconvex}.
	\medskip
	
	\noindent
	{\bf Condition (2) of \cite[Theorem 5.40]{Wise-quasiconvex}.} Observe that pieces in each $\hth{j}$ have length at most $(\gi(\Theta_j)/24)$.
	Hence the condition follows immediately from Lemma~\ref{l:l5.40}.
	\medskip
		
	\noindent
	{\bf Condition (3) of \cite[Theorem 5.40]{Wise-quasiconvex}.} This condition is trivially satisfied since every $\hth{j}$ is finite.
	\medskip
	
	Therefore, by \cite[Theorem 5.40 and Remark 5.41]{Wise-quasiconvex} we conclude that $\hG$ acts metrically properly on the associated CAT(0) cubical complex.
	
	Since $\hG$ is hyperbolic, the cocompactness follows e.g.\ from \cite[Lemma 7.2]{HruskaWise2014}.
\end{proof}

\section{The construction}
\label{s:constr}

Fix $\lambda \in (0,1/24]$ and a natural number $D\geqslant 3$. Let $(\Theta,l)=(\Theta_i,l_i)_{i=1}^{\infty}$ be a sequence of 
$D$--regular graphs 
with a labelling $l$ satisfying the following stronger version of $C'(\lambda)$--small cancellation:
every path in $\Theta_i$ of length greater or equal to $\lambda \mr{girth}(\Theta_i)$ has
labelling different from any other path. Such sequences are constructed in \cite{Osajda-Monster}.

We will construct a sequence $(\hthz,\widehat l)=(\hth{i},\widehat l_i)_{i=1}^{\infty}$ of normal covers of $(\Theta_i,l_i)_{i=1}^{\infty}$ with
the induced labelling $\widehat{l}$.
By $G_i$ we will denote the  finitely presented group given by the graphical presentation $G_i=\langle S \; | \; \hth{1}, \, \hth{2} , \ldots, \hth{i} \rangle$. 
The associated quotient maps will be $\qij{i}{j}\colon G_i \twoheadrightarrow G_j$, with $\qij{i}{i+1}$ denoted $q_i$. At the same time we will construct 
maps to finite groups $\phij{i}{j}\colon G_i \to F_j$, for $i\geqslant j$ with $\phij{i}{i}$ denoted $\phij{i}{}$. 
We will denote $G=\varinjlim (G_i,\qij{i}{j})$, with $\qij{i}{\infty} \colon G_i \to G$, and $\phij{\infty}{i} \colon G \to F_i$ being the induced maps. 

We require that the labelled graphs $(\hth{i},\widehat l_i)_{i=1}^{\infty}$, and the  maps $\phij{i}{j}\colon G_i \to F_j$ satisfy the following conditions:
\begin{enumerate}
	\item[{$(A)$}] $(\hth{1},\htl{1}),(\hth{2},\htl{2}),\ldots$ is a $C'(\lambda)$--small cancellation labelling;
	\item[{$(B)$}] $(\hth{1},\htl{1}),(\hth{2},\htl{2}),\ldots,(\hth{i},\htl{i})$
	admits a proper $\mathbb{Z}_2$--walling for every $i$;
	\item[{$(C)$}] $\phij{j}{}(g)\neq 1$, for every $j$ and every $g\in B_{j}(1,\mr{Cay}(G_{j},S))\setminus \{1\};$	
	\item[$(D)$] $\phij{l}{j}\circ \qij{k}{l}=\phij{k}{j}$, for all $j\leqslant k \leqslant l$.
\end{enumerate}
In particular, the diagram below is commutative.

\begin{figure}[h!]
	\centering
	\includegraphics[width=1\textwidth]{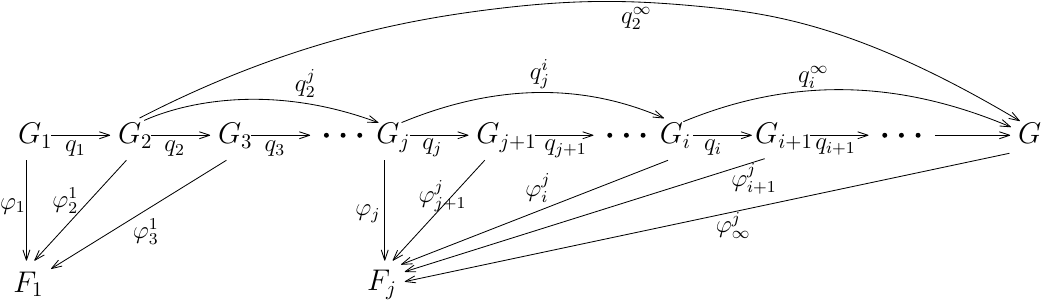}
\end{figure}

We construct the graphs $\hth{i}$, the finite groups $F_i$, and the maps $\phij{i}{j}$ ($j\leqslant i$) inductively,
with respect to $i$. 

\subsection{Induction basis}
\label{s:step0}
Let $\hth{1}$ be the $\mathbb{Z}_2$--homology cover of $\Theta_1$ (such a cover corresponds to a characteristic subgroup of $\pi_1(\Theta_1)$).
Then the following conditions are satisfied:
\begin{enumerate}
	\item[{$(A_1)$}] $(\hth{1},\htl{1}),(\Theta_{2},l_2), (\Theta_3,l_{3}), \ldots$ is a $C'(\lambda)$--small cancellation labelling;
	\item[{$(B_{1})$}] $(\hth{1},\htl{1})$
	admits a proper $\mathbb Z_2$--walling $\cW{1}$.
\end{enumerate}
Then $G_1:=\langle S \; | \; \hth{1} \rangle$ is hyperbolic and, by Lemma~\ref{l:c0cc},
acts geometrically on a CAT(0) cubical complex. Therefore, by results of 
Wise \cite{Wise-quasiconvex} and Agol \cite{Agol2013} it is residually finite.
Let $\phij{1}{}\colon G_1 \to F_1$ be a map into a finite group $F_1$ such that:
\begin{enumerate}
	\item[{$(C_1)$}] $\phij{1}{}(g)\neq 1$ for all $g\in B_1(1,\mr{Cay}(G_1,S))\setminus \{1\}$. 
\end{enumerate}

\subsection{Inductive step}
\label{s:step}
Assume that the graphs $\hth{1},\hth{2},\ldots,\hth{i}$,
the finite groups $F_1,F_2,\ldots,F_i$, and the maps $\phij{j}{k}\colon G_j \to F_k$, for $k\leqslant j \leqslant i$ with the following properties have been constructed.

\begin{enumerate}
	\item[{$(A_i)$}] $(\hth{1},\htl{1}),\ldots,(\hth{i},\htl{i}),(\Theta_{i+1},l_{i+1}), (\Theta_{i+2},l_{i+2}), \ldots$ is a $C'(\lambda)$--small cancellation labelling;
	\item[{$(B_{i})$}] $(\hth{1},\htl{1}),\ldots,(\hth{i},\htl{i})$
	admits a proper $\mathbb Z_2$--walling;
	\item[{$(C_i)$}] $\phij{j}{}(g)\neq 1$, for every $j\leqslant i$ and every $g\in B_{j}(1,\mr{Cay}(G_{j},S))\setminus \{1\};$	
	\item[$(D_i)$] $\phij{l}{j}\circ \qij{k}{l}=\phij{k}{j}$, for all $j\leqslant k \leqslant l \leqslant i$ (that is, the part of the above diagram with all indexes at most $i$ is commutative).
\end{enumerate}
Note that the condition $(D_1)$ is satisfied trivially.

Let $H_i$ be a subgroup of $G_i$ generated by (the images by $F(S)\twoheadrightarrow G_i$ of) all the words
read along cycles in $(\Theta_{i+1},l_{i+1})$. The subgroup $K_i:=\bigcap_{j\leqslant i}\ker(\phij{i}{j}) \lhd G_i$ is of finite index. Therefore $H_i\cap K_i < H_i$ is of finite index
and we can find a finite normal cover $\ov{\Theta}_{i+1}$ of $\Theta_{i+1}$ such that
the normal closure in $G_i$ of the subgroup generated by words read along $(\ov{\Theta}_{i+1},\ov l_{i+1})$ is contained in
$K_i$, where $\ov l_{i+1}$ is the labelling of $\ov{\Theta}_{i+1}$ induced by $l_{i+1}$ via
the covering map. 
Labelled paths of length greater or equal $\lambda \mr{girth}(\ov{\Theta}_{i+1})$ in $(\ov{\Theta}_{i+1},\ov l_{i+1})$ do not appear in
$(\hth{j},\widehat{l}_j)$ for $j\leqslant i$, neither in $(\Theta_{j},{l}_j)$ for $j\geqslant i+2$.
Any two such paths in $(\ov{\Theta}_{i+1},\ov l_{i+1})$ differ by a covering automorphism.
Therefore, $(\hth{1},\htl{1}),\ldots,(\hth{i},\htl{i}),(\ov{\Theta}_{i+1},\ov l_{i+1}),(\Theta_{i+2},l_{i+2}), (\Theta_{i+3},l_{i+3}), \ldots$ is a 
$C'(\lambda)$--small cancellation labelling.  Let  $\hth{i+1}$ be the $\mathbb Z_2$--homology cover of $\ov{\Theta}_{i+1}$. Then the following properties are satisfied:

\begin{enumerate}
	\item[{$(A_{i+1})$}] $(\hth{1},\htl{1}),\ldots,(\hth{i},\htl{i}),(\hth{i+1},\htl{i+1}),
	(\Theta_{i+2},l_{i+2}), (\Theta_{i+3},l_{i+3}), \ldots$ is \\ a $C'(\lambda)$--small cancellation labelling;
	\item[{$(B_{i+1})$}] $(\hth{1},\htl{1}),\ldots,(\hth{i},\htl{i}),(\hth{i+1},\htl{i+1})$
	admits a proper $\mathbb Z_2$--walling,
\end{enumerate}
where $q_i\colon G_i\twoheadrightarrow G_{i+1}:=\langle S \; |\;(\hth{1},\htl{1}),\ldots,(\hth{i+1},\htl{i+1})\rangle$ is the quotient map.
Observe that we have $\ker(q_i)< K_i$.

The group $G_{i+1}=\langle S \; | \; \hth{1}, \, \hth{2} , \ldots, \hth{i+1} \rangle$ is hyperbolic
and, by Lemma~\ref{l:c0cc}, acts geometrically on a CAT(0) cubical complex. Hence, by results of Wise \cite{Wise-quasiconvex} and Agol \cite{Agol2013}, it is residually finite.
Therefore,
we find a map $\varphi_{i+1}\colon G_{i+1}\to F_{i+1}$ into a finite group $F_{i+1}$ such that 
$\phij{i+1}{}(g)\neq 1$ for all $g\in B_{i+1}(1,\mr{Cay}(G_{i+1},S))$. Hence, by $(C_i)$, we have

\begin{enumerate}
		\item[{$(C_{i+1})$}] $\phij{j}{}(g)\neq 1$, for every $j\leqslant i+1$ and every $g\in B_{j}(1,\mr{Cay}(G_{j},S))\setminus \{1\}.$	
\end{enumerate}
For $j\leqslant i$ we define $\phij{i+1}{j}\colon G_{i+1}\to F_j$ as $\phij{i+1}{j}(\qij{j}{i+1}(g))=
\phij{j}{}(g)$. 
This is a well defined homomorphism: If $\qij{j}{i+1}(g)=\qij{j}{i+1}(g')$ then
$\qij{j}{i}(gg'^{-1})\in \ker(q_i)$, and hence
\begin{align*}
\phij{i+1}{j}(\qij{j}{i+1}(g))(\phij{i+1}{j}(\qij{j}{i+1}(g')))^{-1}&=
\phij{j}{}(g)\phij{j}{}(g')^{-1}=\\
&=\phij{j}{}(gg'^{-1})=\phij{i}{j}(\qij{j}{i}(gg'^{-1}))=1, 
\end{align*}
by $\ker(q_i)<K_i$.
By $(D_i)$ and the definition of $\phij{i+1}{j}$ we have:
\begin{enumerate}
	\item[{$(D_{i+1})$}] $\phij{l}{j}\circ \qij{k}{l}=\phij{k}{j}$, for all $j\leqslant k \leqslant l \leqslant i+1$ (that is, the part of the above diagram with all indexes at most $i+1$ is commutative).
\end{enumerate}
This finishes the inductive step.

\subsection{Proof of Main Theorem}
\label{s:proof}
The presentation $\langle S \; | \; \hth{1}, \, \hth{2} , \ldots \rangle$ is a graphical $C'(\lambda)$--small cancellation presentation, by $(A)$. Thus,
the Cayley graph $\mr{Cay}(G,S)$ contains isometrically embedded copies of all the graphs
$\hth{i}$, by Lemma~\ref{l:Gromov}. That is, $\mr{Cay}(G,S)$ contains a sequence of $D$--regular graphs of growing girth, and hence
$G$ is non-exact, by \cite{Willett2011}. 

We show now that $G$ is residually finite.
Take a non trivial element $g\in G$. Let $i$ be such an integer that $g\in B_{i}(1,\mr{Cay}(G,S))$.
Then there exists $g'\in G_i$ such that 
$\qij{i}{\infty}(g')=g$, and $g'\in B_{i}(1,\mr{Cay}(G_i,S))$.
For the homomorphism
$\phij{\infty}{i}\colon G \to F_i$ into the finite group $F_i$ we have 
$\phij{\infty}{i}(g)=\phij{\infty}{i}\circ \qij{i}{\infty}(g')=\phij{i}{}(g')\neq 1$, by $(D)$ and $(C)$.
This shows that $G$ is residually finite.


\bibliography{mybib}{}
\bibliographystyle{plain}

\end{document}